\documentclass[reqno]{amsart}

\usepackage{amsmath,amsthm,amssymb,amsfonts,enumerate,graphicx,color,pstricks}

\usepackage[utf8]{inputenc}
\numberwithin{equation}{section} 
\theoremstyle{plain}
\newtheorem{theo+}           {Theorem}      [section]
\newtheorem{prop+}  [theo+]  {Proposition}
\newtheorem{coro+}  [theo+]  {Corollary}
\newtheorem{lemm+}  [theo+]  {Lemma}
\newtheorem{defi+}  [theo+]  {Definition}
\newtheorem{conj+}  [theo+]  {Conjecture}

\theoremstyle{definition}
\newtheorem{rema+}  [theo+]  {Remark}
\newtheorem{prob+}  [theo+]  {Problem}
\newtheorem{exam+}  [theo+]  {Example}

\newenvironment{proposition}{\begin{prop+}}{\end{prop+}}
\newenvironment{corollary}{\begin{coro+}}{\end{coro+}}
\newenvironment{lemma}{\begin{lemm+}}{\end{lemm+}}

\newcommand{\ti}{\mathrm i}

\begin{document}

\baselineskip 18pt
\larger[2]
\title
[Rahman's biorthogonal functions and superconformal indices] 
{Rahman's biorthogonal rational functions\\ and superconformal indices}
\author{Hjalmar Rosengren}
\address
{Department of Mathematical Sciences
\\ Chalmers University of Technology and University of Gothenburg\\SE-412~96 G\"oteborg, Sweden}
\email{hjalmar@chalmers.se}
\urladdr{http://www.math.chalmers.se/{\textasciitilde}hjalmar}

\thanks{Supported by the Swedish Science Research
Council (Vetenskapsr\aa det). }

\dedicatory{\Large Dedicated to the memory of Mizan Rahman}

\begin{abstract}
We study biorthogonal functions related to basic hypergeometric integrals
with coupled continuous and discrete components. Such integrals appear as superconformal indices for three-dimensional quantum field theories and also in the context of solvable lattice models. We obtain explicit biorthogonal systems given by products of two of Rahman's biorthogonal rational ${}_{10}W_9$-functions or their degenerate cases. We also give new bilateral extensions of the Jackson and $q$-Saalsch\"utz summation formulas and new continuous and discrete biorthogonality measures for Rahman's functions.
\end{abstract}

\maketitle

\section{Introduction}  

During its long history, the notion of
\emph{classical orthogonal polynomials} has gradually been expanded
 to include more general systems. The  most restrictive definition includes only Jacobi, Laguerre and Hermite polynomials (with special cases such as Chebyshev and Legendre polynomials). By Sonine's theorem, these are precisely the orthogonal polynomials whose derivatives are again orthogonal. After the work of Hahn and others in the early 20th century, it appeared natural to include also systems related to difference rather than differential equations. This development culminated in the Askey scheme of hypergeometric and basic hypergeometric orthogonal polynomials  \cite{ks}, consisting of Askey--Wilson and $q$-Racah polynomials together with numerous degenerate cases.

Jacobi polynomials are orthogonal with respect to the measure in Euler's beta integral
$$\int_{-1}^1(1-x)^{a-1}(1+x)^{b-1}\,dx=2^{a+b-1}\frac{\Gamma(a)\Gamma(b)}{\Gamma(a+b)}. $$
More generally,  the Askey scheme is related to a scheme  of 
integral evaluations and summations, many of which are also called beta integrals. From this perspective, the Askey scheme appears to be incomplete. For instance, the  $q$-Racah polynomials correspond to Rogers' ${}_6W_5$-summation 
\cite[Eq.\ (II.21)]{gr}, which is  a degenerate case of Jackson's ${}_8W_7$-summation  \cite[Eq.\ (II.22)]{gr}. Thus,  the top level in the Askey scheme corresponds to the next-to-top level in the scheme of beta integrals.

Mizan Rahman and Jim Wilson realized that to find 
the missing level in the Askey scheme it is necessary to extend its scope from orthogonal polynomials to 
 \emph{biorthogonal rational functions}. To our knowledge, the earliest occurrence of hypergeometric biorthogonal 
rational functions is in \cite{ra0}, dealing with functions of type ${}_3F_2$ and ${}_4F_3$.
At the end of the paper, Rahman remarks that Wilson  independently had found more general discrete and continuous systems.
This seems to refer to the missing level of the classical part of the Askey scheme, situated above Wilson and Racah polynomials. 
In \cite{r1}, Rahman gave a system of biorthogonal rational functions of type ${}_{10}W_9$, generalizing Askey--Wilson polynomials. 
Again, he refers to at the time unpublished work of Wilson, 
who found the corresponding discrete system, generalizing $q$-Racah polynomials \cite{w}.
The complete extension of the Askey scheme to biorthogonal rational functions
has  been worked out  (in the $q$-case) by van de Bult and Rains \cite{br}.

The work of Rahman and Wilson indicates that ${}_{10}W_9$ is the most general ``classical" hypergeometric function. It must have been a great surprise to the special functions community when even more general functions turned up in the work of Date et al.\ \cite{dj} on solvable lattice models. Their elliptic $6j$-symbols  generalized  ${}_{10}W_9$-sums in a completely unexpected direction, providing the first known example of \emph{elliptic hypergeometric functions}; see \cite{re} for an introduction.  

It took some time before the study of elliptic hypergeometric functions gained momentum. 
In one of the first papers on the subject, Spiridonov and Zhedanov \cite{sz} constructed a system of biorthogonal rational functions generalizing Wilson's discrete ${}_{10}W_9$-functions.

A first step towards  elliptic extensions of Askey--Wilson polynomials was
Spiridonov's discovery of the elliptic beta integral \cite{sp1}
\begin{equation}\label{sb}\oint\frac{\prod_{j=1}^6\Gamma(t_jz;p,q)\Gamma(t_j/z;p,q)}{\Gamma(z^2;p,q)\Gamma(z^{-2};p,q)}\frac{dz}{2\pi\ti z}=\frac{2}{(p;p)_\infty(q;q)_\infty}\prod_{1\leq i<j\leq 6}\Gamma(t_it_j;p,q). \end{equation}
Here, the integration is over a contour separating geometric sequences of poles going to infinity from sequences going to zero.
The parameters satisfy the balancing condition $t_1\dotsm t_6=pq$,  the elliptic gamma function is defined by \cite{ru}
$$\Gamma(z;p,q)=\prod_{j,k=0}^\infty\frac{1-p^{j+1}q^{k+1}/z}{1-p^jq^kz} $$
and $(p;p)_\infty=\prod_{j=0}^\infty(1-p^{j+1})$.

In \cite{sp2}, Spiridonov introduced systems biorthogonal with respect to the measure in \eqref{sb}.
As a direct extension of Rahman's work, he  found explicit
   functions $Q_k$ and $R_k$ satisfying a biorthogonality relation $\mu(Q_kR_l)=C_k\delta_{kl}$, where 
$$\mu(f)=\frac{(p;p)_\infty(q;q)_\infty}{2\prod_{1\leq i<j\leq 6}\Gamma(t_it_j;p,q)} \oint f(z)\frac{\prod_{j=1}^6\Gamma(t_jz;p,q)\Gamma(t_j/z;p,q)}{\Gamma(z^2;p,q)\Gamma(z^{-2};p,q)}\frac{dz}{2\pi\ti z}.$$
These functions are elliptic (in a multiplicative coordinate, namely, $Q_k(pz)=Q_k(z)$ and $R_k(pz)=R_k(z)$) and invariant under $z\mapsto z^{-1}$. It follows that they are rational after a change of variables. More generally, Spiridonov found that if 
$\tilde Q_k$ and $\tilde R_k$ denote the functions obtained from $Q_k$ and $R_k$ after interchanging $p$ and $q$, then
\begin{equation}\label{dib}\mu\left(Q_{k_1}\tilde Q_{k_2}R_{l_1}\tilde R_{l_2}\right)=C_{k_1}\tilde C_{k_2}\delta_{k_1l_1}\delta_{k_2l_2}. \end{equation}
This  two-index biorthogonality for a one-variable integral is  quite unusual. It is
not a biorthogonality of rational functions as the functions involved are rational in two different variables.

An intriguing consequence of \eqref{dib} is a property that we will call the \emph{decoupling phenomenon}. We can write \eqref{dib} as
\begin{equation}\label{bm}\mu(fg)=\mu(f)\mu(g),\end{equation}
where  $f=Q_{k_1}R_{l_1}$ and $g=\tilde Q_{k_2}\tilde R_{l_2}$. Taking linear combinations of these basis elements, it follows that \eqref{bm} holds for any functions
$f$ and $g$ satisfying appropriate conditions. 
In particular, $f$ and $g$ should be elliptic with distinct periods ($f(pz)=f(z)$ and $g(qz)=g(z)$) and invariant under $z\mapsto z^{-1}$.  There are also restrictions on the location of their poles.
In the present work we will investigate other instances of the decoupling phenomenon.

Dolan and Osborn \cite{do} discovered that \eqref{sb} (and more general integral identities of Rains \cite{ra}) 
can be interpreted in terms of supersymmetric quantum field theories. 
 It is 
 believed that if two  theories are dual in the sense of Seiberg \cite{se}, then their so called superconformal indices \cite{km, rom} agree. For a particular pair of dual theories, this gives \eqref{sb}. 
Exploiting this idea has led to long lists of elliptic hypergeometric identities, most of which still lack rigorous proof, see e.g.\ \cite{sv1,sv2}.

The quantum field theories related to elliptic hypergeometric integrals live on four-dimensional space-times. In a similar way, one can obtain basic hypergeometric integrals from
three-dimensional theories \cite{gr1,ga,hw,i,kw,kvs,y}. Interestingly, the resulting integrals are not of a type considered in the classical literature but  involve a mixture of continuous and discrete integration. (This can also happen for four-dimensional theories \cite{k,sp3}, but then with a finite rather than infinite discrete component.)
As an example,
a top level  integral of this type is \cite{ga}
\begin{multline}\label{gi}\sum_{x=-\infty}^\infty\oint \frac{(1-q^x z^2)(1-q^xz^{-2})}{q^xz^{6x}}\prod_{j=1}^6\frac{(q^{1+x/2}/b_jz,q^{1-x/2}z/b_j)_\infty}{(q^{N_j+x/2}b_jz,q^{N_j-x/2}b_j/z)_\infty} \frac{dz}{2\pi\ti z}\\
=\frac 2{\prod_{j=1}^6q^{\binom{N_j}2} b_j^{N_j}}\prod_{1\leq i<j\leq 6}
\frac{(q/b_ib_j)_\infty}{(b_ib_jq^{N_i+N_j})_\infty},
\end{multline}
valid for generic parameters  $b_j$ and integer parameters $N_j$ subject to 
 \begin{equation}\label{b}b_1\dotsm b_6=q\end{equation}
 and $N_1+\dots+N_6=0$. 
 The identity \eqref{gi} and some related results can also be interpreted as star-triangle relations for solvable lattice models \cite{gk,gs,k1,k,ya}.
 
The main purpose of the present work is to investigate the ``classical orthogonal polynomials" corresponding to the integral \eqref{gi} and another integral from \cite{ga} (see \eqref{sp} below). It turns out that there is again a decoupling phenomenon such as \eqref{bm}, but with $f$ a rational function of $zq^{-x/2}$ and $g$ a rational function of $zq^{x/2}$ (subject to certain restrictions). 
We can then obtain two-index biorthogonal functions, which in the case of \eqref{gi} are products of Rahman's ${}_{10}W_9$-functions. In the case of \eqref{sp}, there are biorthogonalities 
involving degenerate cases 
  of Rahman's functions due to Al-Salam and Ismail \cite{ai} and van de Bult and Rains \cite{br}.
An intriguing question is whether the decoupling phenomenon and the two-index biorthogonality have any relevance for physics,  in the context of either quantum field theory or solvable lattice models.

The plan of our paper is as follows. In \S \ref{bss} we give two basic hypergeometric summation formulas,
which provide bilateral extensions of the Jackson and $q$-Saalsch\"utz summation. Although these follow easily from known results, they seem not to have been stated explicitly before.
In \S \ref{brs} we discuss Rahman's biorthogonal rational functions. We obtain some new results, in particular, new continuous and discrete biorthogonality measures. Finally, in \S \ref{sis} we  study the decoupling phenomenon and  related two-index biorthogonal systems for integrals arising as superconformal indices.

Throughout, $q$ will be a fixed complex number with $0<|q|<1$. We also fix a choice of square root $q^{1/2}$.
We follow the standard notation of \cite{gr}, but suppress $q$ from the notation. Thus, we write $q$-shifted factorials as
$$(a)_k=\begin{cases}(1-a)(1-aq)\dotsm(1-aq^{k-1}), & k\in\mathbb Z_{\geq 0},\\
\left((1-aq^{-1})(1-aq^{-2})\dotsm(1-aq^k)\right)^{-1}, & k\in\mathbb Z_{<0}\end{cases} $$
and
$$(a)_\infty=\prod_{j=0}^\infty(1-aq^j). $$
We will use the theta function
$$\theta(x)=(x)_\infty(q/x)_\infty, $$
which satisfies the identity
\begin{equation}\label{tqp}\theta(xq^k)=(-1)^kq^{-\binom{k}2}x^{-k}\theta(x),\qquad k\in\mathbb Z. \end{equation}
Repeated arguments of shifted factorials and theta functions stand for products, for instance,
$$(ab^\pm,c)_k=(ab)_k(a/b)_k(c)_k. $$
Finally, we recall the standard basic hypergeometric series
\begin{align*}
{}_{r+1}\phi_r\left(\begin{matrix}a_1,\dots,a_{r+1}\\b_1,\dots,b_r\end{matrix};z\right)
&=\sum_{k=0}^\infty\frac{(a_1,\dots,a_{r+1})_k}{(q,b_1,\dots,b_r)_k}\,z^k\\
{}_{r}\psi_r\left(\begin{matrix}a_1,\dots,a_{r}\\b_1,\dots,b_r\end{matrix};z\right)
&=\sum_{k=-\infty}^\infty\frac{(a_1,\dots,a_{r})_k}{(b_1,\dots,b_r)_k}\,z^k,\\
{}_{r+1}W_{r}(a;b_1,\dots,b_{r-2};z)&=\sum_{k=0}^\infty\frac{1-aq^{2k}}{1-a}\frac{(a,b_1,\dots,b_{r-2})_k}{(q,aq/b_1,\dots,aq/b_{r-2})_k}\,z^k.
\end{align*}

 {\bf Acknowledgements:} 
This work is dedicated to the memory of Mizan Rahman, 
a gentle and generous mathematician 
whose unsurpassed mastery of $q$-series has been a great inspiration for me.
I thank Ilmar Gahramanov for fruitful discussions on superconformal indices and Slava Spiridonov for encouraging me to finish the present paper.
  
\section{Bilateral summations}\label{bss}

\subsection{A bilateral Jackson summation}

Although it is a straight-forward consequence of results that have been known since the 1950s, the following identity seems to fill a gap in the literature.

\begin{proposition}\label{bjp} If $\lambda,\mu,b_1,\dots,b_6$ are generic parameters 
subject to \eqref{b}, 
then
\begin{multline}\label{bj}
\left\{\theta(\mu b_1,\dots,\mu b_6)-\frac{q}{\mu^2}\,\theta(\mu/b_1,\dots,\mu/b_6)\right\}\\
\times\frac{(1-\lambda^2)\prod_{j=1}^6(q\lambda^\pm/b_j)_\infty}{\theta(\mu/\lambda)}
\,{}_8\psi_8\left(\begin{matrix}\lambda q,-\lambda q,\lambda b_1,\dots,\lambda b_6\\
\lambda,-\lambda,\lambda q/b_1,\dots,\lambda q/b_6\end{matrix};q\right)+\text{\rm idem}(\lambda;\mu)\\
=(q)_\infty\theta(\lambda^2,\mu^2,\lambda\mu)\prod_{1\leq i<j\leq 6}(q/b_ib_j)_\infty.
\end{multline}
\end{proposition}

Here, we are using the  notation
$$ f(\lambda,\mu)+\text{idem}(\lambda;\mu)=f(\lambda,\mu)+f(\mu,\lambda).$$
If  we let $\lambda=b_6$ and $\mu=b_1$ in \eqref{bj}, then both ${}_8\psi_8$-series reduce to ${}_8W_7$-series.
The resulting identity can be written
\begin{multline}\label{ns}{}_8W_7(b_6^2;b_1 b_6,\dots,b_5 b_6;q)-\frac{b_1}{b_6}\frac{( qb_6^2,qb_1/b_6)_\infty}{(qb_1^2,qb_6/b_1)_\infty}\prod_{j=2}^5\frac{(qb_1/b_j,b_6b_j)_\infty}{(qb_6/b_j,b_1b_j)_\infty}\\
\times{}_8W_7(b_1^2;b_1 b_2,\dots,b_1 b_6;q) 
=\frac{(qb_6^2,b_1/b_6)_\infty\prod_{2\leq i<j\leq 5}(q/b_ib_j)_\infty}{\prod_{j=2}^5(qb_6/b_j,b_1b_j)_\infty},
\end{multline}
which is the nonterminating  Jackson summation  \cite[Eq.\ (II.25)]{gr}. Consequently, we call 
\eqref{bj} the \emph{bilateral Jackson summation}. 

To prove Proposition \ref{bjp},  we start from Jackson's   transformation \cite{j} 
(due to Margaret Jackson, not Frank Hilton Jackson of the Jackson summation)
\begin{multline*}
\frac{\prod_{j=1}^6(qa^\pm/b_j)_\infty}{(qa^{\pm 2})_\infty}\,{}_8\psi_8\left(\begin{matrix}aq,-aq,ab_1,\dots,ab_6\\
a,-a,aq/b_1,\dots,aq/b_6\end{matrix};\frac{q^2}{b_1\dotsm b_6}\right)\\
=\frac{\theta(\mu a^\pm)\prod_{j=1}^6(q\lambda^\pm/b_j)_\infty}{\theta(\mu\lambda^\pm)(q\lambda^{\pm 2})_\infty}\,{}_8\psi_8\left(\begin{matrix}\lambda q,-\lambda q,\lambda b_1,\dots,\lambda b_6\\
\lambda,-\lambda,\lambda q/b_1,\dots,\lambda q/b_6\end{matrix};\frac{q^2}{b_1\dotsm b_6}\right)+\text{idem}(\lambda;\mu),
 \end{multline*} 
which is a special case of Slater's transformations for
${}_{2r}\psi_{2r}$-series  \cite[Eq.~(5.5.2)]{gr}. We will only need the case $a=b_6$, that is,
\begin{multline}\label{jt}
\frac{(q)_\infty\prod_{j=1}^5(qb_6^\pm/b_j)_\infty}{(qb_6^2)_\infty}\,{}_8W_7(b_6^2;b_1b_6,\dots,b_5 b_6;q^2/b_1\dotsm b_6)\\
=\frac{\theta(\mu b_6^\pm)\prod_{j=1}^6(q\lambda^\pm/b_j)_\infty}{\theta(\mu\lambda^\pm)(q\lambda^{\pm 2})_\infty}\,{}_8\psi_8\left(\begin{matrix}\lambda q,-\lambda q,\lambda b_1,\dots,\lambda b_6\\
\lambda,-\lambda,\lambda q/b_1,\dots,\lambda q/b_6\end{matrix};\frac{q^2}{b_1\dotsm b_6}\right)\\
+\text{idem}(\lambda;\mu).
 \end{multline} 
Assuming \eqref{b}, we may apply \eqref{jt} to both ${}_8W_7$ series in \eqref{ns}, obtaining after simplification
\begin{multline*}
\left\{\theta(b_1b_2,\dots,b_1b_5,\mu b_6^\pm)-\frac{b_1}{b_6}\,\theta(b_2b_6,\dots,b_5b_6,\mu b_1^\pm)\right\}\\
\times\frac{\prod_{j=1}^6(q\lambda^\pm/b_j)_\infty}{\theta(\mu\lambda^\pm)(q\lambda^{\pm 2})_\infty}
\,{}_8\psi_8\left(\begin{matrix}\lambda q,-\lambda q,\lambda b_1,\dots,\lambda b_6\\
\lambda,-\lambda,\lambda q/b_1,\dots,\lambda q/b_6\end{matrix};q\right)+\text{idem}(\lambda;\mu)\\
=(q)_\infty\theta(b_1/b_6)\prod_{1\leq i<j\leq 6}(q/b_ib_j)_\infty.
\end{multline*}
The proof of \eqref{bj} is now reduced to the  theta function identity
\begin{multline}\label{ftt}
\left\{\theta(b_1b_2,\dots,b_1b_5,\mu b_6^\pm)-\frac{b_1}{b_6}\,\theta(b_2b_6,\dots,b_5b_6,\mu b_1^\pm)\right\}\\
=\frac{\theta(b_1/b_6)}{\theta(\mu^2)}\left\{\theta(\mu b_1,\dots,\mu b_6)-\frac q{\mu^2}\,\theta(\mu/b_1,\dots,\mu/b_6)\right\},
\end{multline}
which is a special case of \cite[Ex.\ 20.53.3]{ww}, see also \cite[Lemma 6.4]{ro}.

\subsection{A bilateral $q$-Saalsch\"utz summation}

The analogue of Proposition \ref{bjp} at the level of the Saalsch\"utz summation combines three rather than two bilateral sums. We find it convenient to write the identity as a determinant evaluation.

\begin{proposition}\label{bsp}
For generic parameters subject to
\begin{equation}\label{aba}a_1a_2a_3b_1b_2b_3=q,\end{equation}
let
\begin{align*}
f_1(\lambda)&=\lambda\,\theta(b_1/\lambda,b_2/\lambda,b_3/\lambda),\\
f_2(\lambda)&=\lambda^{-1}\theta(a_1\lambda,a_2\lambda,a_3\lambda),\\
f_3(\lambda)&=\prod_{j=1}^3(q/a_j\lambda,q\lambda/b_j)_\infty\cdot{}_3\psi_3\left(\begin{matrix}a_1\lambda,a_2\lambda,
a_3\lambda\\q\lambda/b_1,q\lambda/b_2,q\lambda/b_3\end{matrix};q\right).
\end{align*}
Then,
\begin{equation}\label{bsi}\det_{1\leq i,j\leq 3}\left(f_j(\lambda_i)\right)= \frac{(q)_\infty\theta(t\lambda_1\lambda_2\lambda_3)}{\lambda_1\lambda_2\lambda_3}\prod_{i,j=1}^3(q/a_ib_j)_\infty\prod_{1\leq i<j\leq 3}\lambda_j\theta(\lambda_i/\lambda_j),\end{equation}
where  
\begin{equation}\label{t}t=a_1a_2a_3=q/{b_1b_2b_3}.\end{equation}
\end{proposition}

One may prove Proposition \ref{bsp} in a similar way as Proposition \ref{bjp}, but we prefer to illustrate a different method 
(which can also be adapted to give an alternative proof of Proposition \ref{bjp}). We first observe that each of the three functions $f_j$  satisfies
$f_j(q\lambda )=-f_j(\lambda)/qt\lambda^3$.
For $f_1$ and $f_2$ this follows from the quasi-periodicity 
$\theta(qx)=-\theta(x)/x$
  and for $f_3$ it is easy to verify by shifting the summation index.
It then follows from \cite[Prop.~3.4]{rs} that \eqref{bsi} holds up to a factor
 independent of the variables $\lambda_j$. Thus, it is enough to verify \eqref{bsi} for fixed values of these variables. We choose
 $\lambda_2=b_2$, $\lambda_3=b_3$, so that $f_1(\lambda_2)=f_1(\lambda_3)=0$. We may then cancel the factors involving $\lambda_1$ and are left with the identity
\begin{multline*}
b_3(a_1b_2,a_2b_2,a_3b_2,b_3q/b_1,b_3q/b_2)_\infty\,{}_3\phi_2\left(\begin{matrix}a_1b_3,a_2b_3,
a_3b_3\\b_3q/b_1,b_3q/b_2\end{matrix};q\right)\\
-\text{idem}(b_2;b_3)=b_3\theta(b_2/b_3)\prod_{j=1}^3(q/b_1a_j)_\infty,
 \end{multline*} 
 which is the nonterminating $q$-Saalsch\"utz summation  \cite[Eq.\ (II.24)]{gr}. This proves \eqref{bsi} and explains why call it the bilateral $q$-Saalsch\"utz summation.

\section{Biorthogonal rational functions}\label{brs}

\subsection{The Rahman functional}\label{rfs}

Let $b_1,\dots,b_6$ be  complex numbers subject to the balancing condition \eqref{b}. We assume that they are generic in the sense that
\begin{equation}\label{gc}b_i/b_j\notin q^{\mathbb Z},\quad 1\leq i< j\leq 6,\qquad b_ib_j\notin q^{\mathbb Z},\quad 1\leq i\leq j\leq 6.
\end{equation} 
Let $V$ denote the vector space of rational functions in the variable $(z+z^{-1})/2$ that are regular except for possible single poles at
\begin{equation}\label{vpp}z^\pm\in b_jq^{\mathbb Z_{<0}},\qquad 1\leq j\leq 6.\end{equation}
 We also require that the elements of $V$ are regular at infinity,
that is, their numerator does not have larger degree than their denominator.

\begin{lemma}\label{sl}
The space  $V$ is spanned by the rational functions
\begin{equation}\label{ss}(b_1z^\pm)_{k_1}\dotsm(b_6z^\pm)_{k_6}, \end{equation}
where $k_j$ are integers such that
\begin{equation}\label{kb}k_1+\dots+k_6=0.\end{equation}
\end{lemma}

\begin{proof}
Given $f\in V$, 
let $n_j$ denote the largest positive integer such that $f$ has a pole at $z=b_jq^{-n_j}$; if there is no such pole we let $n_j=0$. Then,
$$f\left(\frac{z+z^{-1}}2\right)={(b_1z^{\pm})_{-n_1}\dotsm (b_6z^{\pm})_{-n_6}}\,p\left(\frac{z+z^{-1}}2\right), $$
with $p$ a polynomial of degree at most $N=n_1+\dots+n_6$. If we can expand
\begin{equation}\label{nue}p\left(\frac{z+z^{-1}}2\right)=\sum_{\substack {l_1,\dots,l_6\geq 0,\\l_1+\dots+l_6=N}}C_l\prod_{j=1}^6 (b_jq^{-n_j}z^{\pm})_{l_j}, 
\end{equation}
      then
 $$ f\left(\frac{z+z^{-1}}2\right)=\sum_{\substack {k_1,\dots,k_6\geq 0,\\k_1+\dots+k_6=0}}C_{k+n}\prod_{j=1}^6 (b_jz^{\pm})_{k_j}$$
 and we are done.
 
 We claim that there is an expansion of the form \eqref{nue} with $l_3=\dots=l_6=0$. Indeed, 
 assuming the condition \eqref{gc}, the polynomials  
 $$(b_1q^{-n_1}z^{\pm})_{l}(b_2q^{-n_2}z^{\pm})_{N-l},\qquad l=0,\dots,N$$
  form a basis for the space of polynomials of degree at most $N$, see e.g.\ 
 \cite[Lemma 3.1]{ro}.
\end{proof}

We now recall Rahman's  integral evaluation
\begin{equation}\label{ri} \oint\frac{(z^{\pm 2})_\infty\theta(\lambda z^\pm)}{(b_1z^{\pm},\dots,b_6z^{\pm})_\infty}\,\frac{dz}{2\pi\ti z}
=\frac{2\left\{\theta(\lambda b_1,\dots,\lambda b_6)-q\lambda^{-2}\theta(\lambda/b_1,\dots,\lambda/b_6)\right\}}{(q)_\infty\theta(\lambda^2)\prod_{1\leq i<j\leq 6}(b_ib_j)_\infty},\end{equation}
where \eqref{b} is assumed. The integration is over a positively oriented contour separating the poles at
$b_jq^{\mathbb Z_{\geq 0}}$ from their reciprocals. 
The identity \eqref{ri} is a special case of an integral from \cite{r1} (see  \cite[Ex.\ 6.7]{gr}), which was explicitly stated in
\cite{r2}. 

Let us replace $b_j$ in \eqref{ri} with $b_jq^{k_j}$, where $k_j$ are integers subject to \eqref{kb}.
Using \eqref{tqp} and other elementary facts, the resulting identity can be expressed as
\begin{multline*}\oint\frac{(z^{\pm 2})_\infty\theta(\lambda z^\pm)}{(b_1z^{\pm},\dots,b_6z^{\pm})_\infty}\,(b_1z^\pm)_{k_1}\dotsm (b_6z^\pm)_{k_6}\,\frac{dz}{2\pi\ti z}\\
=\frac{2\left\{\theta(\lambda b_1,\dots,\lambda b_6)-q\lambda^{-2}\theta(\lambda/b_1,\dots,\lambda/b_6)\right\}}{(q)_\infty\theta(\lambda^2)\prod_{1\leq i<j\leq 6}(b_ib_j)_\infty}\frac{\prod_{1\leq i<j\leq 6}(b_ib_j)_{k_i+k_j}}{\prod_{j=1}^6q^{\binom{k_j}2}b_j^{k_j}}. \end{multline*}
By Lemma \ref{sl}, this proves the following result.

\begin{proposition}
There exists a functional $\mathbf J$ on the space $V$ such that
\begin{equation}\label{id}\mathbf J\left((b_1z^\pm)_{k_1}\dotsm(b_6z^\pm)_{k_6}\right)=\frac{\prod_{1\leq i<j\leq 6}(b_ib_j)_{k_i+k_j}}{\prod_{j=1}^6q^{\binom{k_j}2}b_j^{k_j}}. \end{equation}
For generic values of $\lambda$, it is given by
\begin{multline}\label{ii}\mathbf J(f)=\frac{(q)_\infty\theta(\lambda^2)\prod_{1\leq i<j\leq 6}(b_ib_j)_\infty}{2\left\{\theta(\lambda b_1,\dots,\lambda b_6)-q\lambda^{-2}\theta(\lambda/b_1,\dots,\lambda/b_6)\right\}}\\
\times \oint f\left(\frac{z+z^{-1}}2\right)\frac{(z^{\pm 2})_\infty\theta(\lambda z^\pm)}{(b_1z^{\pm},\dots,b_6z^{\pm})_\infty}\,\frac{dz}{2\pi\ti z},\end{multline}
where the integration is over a positively oriented contour encircling all  poles of the integrand of the form $z\in b_j q^{\mathbb Z}$ and no poles of
the form $z\in b_j^{-1} q^{\mathbb Z}$.
\end{proposition}

Note that the choice of contour in \eqref{ii} depends on $f$.  If $n_j$ are defined as in the proof of Lemma \ref{sl}, then the points
$z=b_jq^{k-n_j}$, $k\in\mathbb Z_{\geq 0}$, should be inside the contour of integration and their reciprocals outside.
For fixed $f$, 
the existence of such a contour follows from  \eqref{gc}, but no contour works for all $f$.

We will refer to $\mathbf J$ as \emph{Rahman's functional}. Rahman preferred to work with \eqref{ii} in the special case $\lambda=b_6$. This yields the more compact but less symmetric expression
\begin{equation}\label{rii}\mathbf J(f)=\frac{(q)_\infty\prod_{1\leq i<j\leq 5}(b_ib_j)_\infty}{2\prod_{j=1}^5(b_1\dotsm b_5/b_j)_\infty}\oint f\left(\frac{z+z^{-1}}2\right)\frac{(z^{\pm 2},b_1\dotsm b_5 z^\pm)_\infty}{(b_1z^{\pm},\dots,b_5z^{\pm})_\infty}\,\frac{dz}{2\pi\ti z}.\end{equation}

We will also obtain expressions for $\mathbf J$ in terms of discrete measures. To this end, replace $b_j$ in \eqref{bj} by $b_jq^{k_j}$
where, as before, we assume \eqref{kb}. The resulting identity can be written 
\begin{multline}\label{ai}
\frac{(1-\lambda^2)\left\{\theta(\mu b_1,\dots,\mu b_6)-{q}{\mu^{-2}}\,\theta(\mu/b_1,\dots,\mu/b_6)\right\}\prod_{j=1}^6(q\lambda^\pm/b_j)_\infty}{(q)_\infty\,\theta(\lambda^2,\mu^2,\mu\lambda^\pm)\prod_{1\leq i<j\leq 6}(q/b_ib_j)_\infty}\\
\times
\sum_{x=-\infty}^\infty\frac{1-\lambda^2q^{2x}}{1-\lambda^2}q^x\prod_{j=1}^6\frac{(\lambda b_j)_x}{(q\lambda/b_j)_x}\,(\lambda b_jq^x,b_jq^{-x}/\lambda)_{k_j}
\\
+\text{\rm idem}(\lambda;\mu)=\frac{\prod_{1\leq i<j\leq 6}(b_ib_j)_{k_i+k_j}}{\prod_{j=1}^6q^{\binom{k_j} 2}b_j^{k_j}}.
\end{multline}
This proves the following result.

\begin{proposition}\label{dip}
For generic values of $\lambda$ and $\mu$, 
\begin{multline}\label{di}
\mathbf J(f)=
\frac{(1-\lambda^2)\left\{\theta(\mu b_1,\dots,\mu b_6)-\frac{q}{\mu^2}\,\theta(\mu/b_1,\dots,\mu/b_6)\right\}\prod_{j=1}^6(q\lambda^\pm/b_j)_\infty}{(q)_\infty\theta(\lambda^2,\mu^2,\lambda\mu,\mu/\lambda)\prod_{1\leq i<j\leq 6}(q/b_ib_j)_\infty}\\
\times
\sum_{x=-\infty}^\infty\frac{1-\lambda^2q^{2x}}{1-\lambda^2}q^x\prod_{j=1}^6\frac{(\lambda b_j)_x}{(q\lambda/b_j)_x}\,f\left(\frac{\lambda q^x+\lambda^{-1}q^{-x}}2\right)
+\text{\rm idem}(\lambda;\mu).
\end{multline}
\end{proposition}

In contrast to \eqref{ii}, where the contour of integration depends on $f$, the identity \eqref{di} holds uniformly on the whole space $V$.

Assume for simplicity that $f((z+z^{-1})/2)$ has no poles at
$z^\pm\in b_5q^{\mathbb Z_{<0}}$ or $b_6q^{\mathbb Z_{<0}}$. Then, we may let  $\lambda=b_5$ and $\mu=b_6$
in \eqref{di} and conclude that
\begin{multline}\label{dis}
\mathbf J(f)=
\frac{\prod_{j=1}^4(qb_5/b_j,b_jb_6)_\infty}{(qb_5^2,b_6/b_5)_\infty\prod_{1\leq i<j\leq 4}(q/b_ib_j)_\infty}\\
\times\sum_{x=0}^\infty\frac{1-b_5^2q^{2x}}{1-b_5^2}q^x\prod_{j=1}^6\frac{(b_5b_j)_x}{(qb_5/b_j)_x}
\,f\left(\frac{b_5q^x+b_5^{-1}q^{-x}}2\right)
+\text{\rm idem}(b_5;b_6).
\end{multline}
For general functions in $V$ one obtains in the same way an expression involving also residues of $f$ at points corresponding to $x<0$. This is discussed in \cite[Prop.~6.4]{br}, but the general case of \eqref{di} 
seems to be new.

Van de Bult and Rains \cite{br} also found integral formulas for $\mathbf J$ that are non-symmetric in the sense that the measure is not invariant under $z\mapsto 1/z$. 
We will now give a more general result of this type.

\begin{proposition}
The functional $\mathbf J$ can be expressed as
\begin{equation}\label{nsj}\mathbf J(f)=\frac{1}{C}
 \oint f\left(\frac{z+z^{-1}}2\right)\frac{z-z^{-1}}{z^2}\prod_{j=1}^6\frac{\theta(\lambda_jz)}{(b_jz^\pm)_\infty}\,\frac{dz}{2\pi\ti z}, \end{equation}
where the integration is as above, $\lambda_j$ are generic parameters subject to $\lambda_1\dotsm\lambda_6=q$ and
$$C=\frac{\prod_{j=1}^4\theta(\lambda_j\lambda_2)\left\{\prod_{j=1}^6\theta(\lambda_1b_j)-{q}{\lambda_1^{-2}}\prod_{j=1}^6\theta(\lambda_1/b_j)\right\}}{\lambda_2\theta(\lambda_1^2,\lambda_1/\lambda_2)(q)_\infty\prod_{1\leq i<j\leq 6}(b_ib_j)_\infty}+\textup{idem}(\lambda_1;\lambda_2). $$
\end{proposition}

\begin{proof}
Symmetrizing the integral, we may write the right-hand  side of \eqref{nsj} as
$$\frac{1}{2C}
 \oint f\left(\frac{z+z^{-1}}2\right)\frac{(z-z^{-1})\left\{z^{-2}\prod_{j=1}^6\theta(\lambda_jz)-z^{2}\prod_{j=1}^6\theta(\lambda_j/z)\right\}}{\prod_{j=1}^6(b_jz^\pm)_\infty}\,\frac{dz}{2\pi\ti z}. $$
 By \eqref{ftt}, with $\mu$ replaced with $z$ and $b_j$ with an appropriate permutation of $\lambda_j$, the factor in brackets can be expressed as 
 $$\frac{z^{-1}\theta(z^2,\lambda_2\lambda_3,\dots,\lambda_2\lambda_6,\lambda_1z^\pm)}{\lambda_2\theta(\lambda_1/\lambda_2)}+\textup{idem}(\lambda_1;\lambda_2). $$
 As
 $$z^{-1}(z-z^{-1})\theta(z^2)=(z^{\pm 2})_\infty, $$
 we may then apply \eqref{ii} to express the integral in terms of $\mathbf J$.
\end{proof}

The non-symmetric integral formulas given in \cite{br} correspond to the special case of \eqref{nsj} when $\lambda_j=b_j$ for $j=1,\dots,4$. In that case, 
\begin{align*}C&=\frac{\theta(b_1b_2,b_1b_3,b_1b_4,b_2b_3,b_2b_4)}{b_2\theta(b_1/b_2)(q)_\infty\prod_{1\leq i<j\leq 4}}\\
&\times\left\{\theta(b_2\lambda_5,b_2\lambda_6,b_1b_5,b_1b_6)-\theta(b_1\lambda_5,b_1\lambda_2,b_2b_5,b_2b_6)\right\}. \end{align*}
Since $b_5b_6=\lambda_5\lambda_6$ we may apply Weierstrass' identity (see e.g.\ \cite[\S 1.4]{re})
to conclude that the factor in bracket equals
$$-b_2\lambda_6\theta(b_1/b_2,b_3b_4,\lambda_5/b_5,\lambda_5/b_6). $$
In the corresponding specialization of \eqref{nsj}, we find it convenient to replace $z$ with $1/z$
 and write $\lambda=q/\lambda_5$. We then obtain
 \begin{align}\nonumber\mathbf J(f)&=\frac{b_5b_6\lambda\,(q,b_5b_6)_\infty\prod_{j=1}^4(b_jb_5,b_jb_6)_\infty}{\theta(b_5\lambda,b_6 \lambda)\prod_{1\leq i<j\leq 4}(q/b_ib_j)_\infty}\\
 \label{brn}&\quad\times\oint  f\left(\frac{z+z^{-1}}2\right)(z-z^{-1})\frac{\theta(\lambda z,z/b_5b_6\lambda)}{(b_5z^\pm,b_6z^\pm)_\infty}\prod_{j=1}^4\frac{(qz/b_j)_\infty}{(b_jz)_\infty}\,\frac{dz}{2\pi\ti z}. \end{align}
This is also contained in \cite[Prop.\ 6.3]{br}. As is discussed in \cite{br}, 
\eqref{dis} can be obtained from \eqref{brn} through shrinking the
 contour 
of integration to zero, picking up residues at the points $b_5q^{\mathbb Z_{\geq 0}}$ 
and $b_6q^{\mathbb Z_{\geq 0}}$.

\subsection{Rahman's biorthogonal functions}

Rahman \cite{r1,r3} constructed rational functions that are biorthogonal with respect to the functional $\mathbf J$.  To explain this, we introduce the subspace $V_n^i$ of $V$ consisting of rational functions 
having no poles except possibly at $z^\pm=b_iq^{-k-1}$, where $i=1,\dots,6$ is fixed and $0\leq k\leq n$.
It is easy to see that $\dim(V_n^i)=n+1$. If $1\leq i\neq j\leq 6$, we may look for a function $Q_n^{(i,j)}\in V_n^i$ such that
$I(Q_n^{(i,j)}f)=0$ for all $f\in V_{n-1}^j$.  As this is an $n$-dimensional  condition on an $(n+1)$-dimensional space, it is not surprising that for generic parameters the solution is unique up to normalization. We then have the biorthogonality relations
$\mathbf J(Q_n^{(i,j)}Q_m^{(j,i)})=0$ for $m\neq n$. 

We may assume that $i=5$, $j=6$. Then, $Q_n^{(5,6)}$ is given by the function (our notation differs from that of Rahman)
\begin{multline*}Q_n\left(\frac{z+z^{-1}}2;b_1,\dots,b_6;q\right)
=\frac{(b_1b_2,b_1b_3,b_1b_4,1/b_1b_6)_n}{(qb_1/b_5)_n}\\
\times\,{}_{10}W_9(b_1/b_5;b_1z,b_1/z,q/b_2b_5,q/b_3b_5,q/b_4b_5,q^n/b_5b_6,q^{-n};q).
\end{multline*}
It follows from Bailey's ${}_{10}W_9$-transformation \cite[Eq.\ (III.28)]{gr} that $Q_n$ is symmetric in the parameters $b_1,\dots,b_4$.

More precisely, Rahman proved that if
\begin{subequations}\label{rqr}
\begin{align}
Q_n&=Q_n\left(\frac{z+z^{-1}}2;b_1,b_2,b_3,b_4,b_5,b_6\right),\\
R_n&=Q_n\left(\frac{z+z^{-1}}2;b_1,b_2,b_3,b_4,b_6,b_5\right), \end{align}
\end{subequations}
then 
\begin{equation}\label{rbr}\mathbf J(Q_mR_n)=\delta_{mn}\frac{1-q^{-1}b_1b_2b_3 b_4}{1-q^{2n-1}b_1b_2b_3b_4}\frac{(q)_n\prod_{1\leq i<j\leq 4}(b_ib_j)_n}{q^n(q^{-1}b_1b_2b_3b_4)_n}. \end{equation}
Rahman formulated this result with $\mathbf J$ given by  \eqref{rii} (and modifications, where the contour is deformed to the unit circle and then moved to an interval  by a change of variables).
Using instead \eqref{ii},
\eqref{di} or \eqref{nsj} gives new explicit forms of \eqref{rbr}. Note, in particular,
that \eqref{di} gives biorthogonality relations with respect to a fixed measure for all $m$ and $n$,
whereas in \eqref{ii}, \eqref{rii} and \eqref{nsj} the integrals must be deformed (or  discrete terms added) as $m$ and $n$ increase.

\subsection{The Al-Salam--Ismail functional}
There are many interesting limit cases of Rahman's biorthogonal functions \cite{br}.
From the perspective of superconformal indices (see \S \ref{sis}), a particularly interesting case is when three of the parameters $b_j$ tend to zero and the remaining three to infinity, while the variable $z$ is scaled so that either $zb_j$ or $z/b_j$ is fixed for each $j$.
To describe this limit, we  make the substitutions
$$(b_1,b_2,b_3,b_4,b_5,b_6)\mapsto(a_1/c,a_2/c,a_3/c,b_1c,b_2c,b_3c), $$
so that the balancing condition  \eqref{b} becomes \eqref{aba}. We also replace $z$ by $cz$.  The possible poles \eqref{vpp}
are then located at
\begin{equation}\label{wpp}z\in a_j^{-1}q^{\mathbb Z_{>0}},\quad z\in b_jq^{\mathbb Z_{<0}},\qquad 1\leq j\leq 3,\end{equation}
together with additional poles tending to zero or infinity as $c$ tends to zero. 

More concretely, let
$f((z+z^{-1})/2)$ denote the spanning function \eqref{ss}. We replace $(k_1,\dots,k_6)$ by $(k_1,k_2,k_3,l_1,l_2,l_3)$ and write
\begin{equation}\label{klb}T=k_1+k_2+k_3=-l_1-l_2-l_3. \end{equation}
 Then,
\begin{align*}
\lim_{c\rightarrow 0}c^{2T}f\left(\frac{(cz)+(cz)^{-1}}2\right)&=\lim_{c\rightarrow 0}\prod_{j=1}^3
c^{2k_j}(a_jz,a_j/c^2z)_{k_j}(b_j/z,b_jc^2z)_{l_j}\\
&=(-1)^T\prod_{j=1}^3q^{\binom{k_j}2}a_j^{k_j}\cdot g(z),
\end{align*}
where
\begin{equation}\label{g}g(z)=z^{-T}\prod_{j=1}^3(a_jz)_{k_j}(b_j/z)_{l_j}.\end{equation}
It is easy to see that, as $k_j$, $l_j$ and $T$ vary over integers subject to \eqref{klb}, the functions $g$
span the space of rational functions in $z$ that are regular everywhere (including at infinity) except for possible single poles at \eqref{wpp}.
We will denote this space by $W$. We also compute
\begin{align*}
\lim_{c\rightarrow 0}c^{2T}\mathbf J(f)&=\lim_{c\rightarrow 0}\frac{c^{2T}\prod_{1\leq i<j\leq 3}(a_ia_j/c^2)_{k_i+k_j}(b_ib_jc^2)_{l_i+l_j}\prod_{i,j=1}^3(a_ib_j)_{k_i+l_j}}{\prod_{j=1}^3q^{\binom {k_j}2+\binom{l_j}2}(a_j/c)^{k_j}(b_jc)^{l_j}}\\
&=t^Tq^{\binom T2}\frac{\prod_{i,j=1}^3(a_ib_j)_{k_i+l_j}}{\prod_{j=1}^3q^{\binom{l_j}2}b_j^{l_j}},
\end{align*}
where $t$ is as in \eqref{t}.
Thus, $\mathbf J$ degenerates to a functional  $\mathbf K$ on $W$, which acts on the spanning functions \eqref{g} as
\begin{equation}\label{kds}\mathbf K(g)=\frac{(-t)^Tq^{\binom T2}\prod_{i,j=1}^3(a_ib_j)_{k_i+l_j}}{\prod_{j=1}^3q^{\binom{k_j}2+\binom{l_j}2}a_j^{k_j}b_j^{l_j}}.\end{equation}
As we explain in \S \ref{grss}, Al-Salam and Ismail constructed explicit biorthogonal systems for the functional $\mathbf K$ \cite{ai}. For this reason, we call it the Al-Salam--Ismail functional.

The reader may find it puzzling
 that $\mathbf J$ and its degenerate case $\mathbf K$ seem to depend on 
the same number of parameters. The explanation is that we can simultaneously scale $z$, $1/a_j$ and $b_j$ by the same constant without changing $\mathbf K$. This effectively reduces the number of parameters by one. We also mention the symmetry
\begin{equation}\label{kis}\tilde{\mathbf K}(\tilde g)=\mathbf K(g),\qquad g\in W, \end{equation}
where $\tilde g(z)=g(1/z)$ and $\tilde{\mathbf K}$ is obtained from $\mathbf K$ after interchanging the parameters $a_j$ and $b_j$.

One may obtain an integral expression for $\mathbf K$ from Gasper's identity
\cite{g}, \cite[Eq.\ (4.11.3)]{gr}
\begin{equation}\label{gci}\oint\frac{(qz/b_3)_\infty\theta(\lambda z,qz/\lambda b_1b_2)}{(a_1z,a_2z,a_3z,b_1/z,b_2/z)_\infty}\frac{dz}{2\pi\ti z} 
=\frac{\theta(b_1\lambda,b_2\lambda)}{(q)_\infty}\prod_{j=1}^3\frac{(q/b_3a_j)_\infty}{(b_1a_j,b_2a_j)_\infty},
\end{equation}
which holds for generic parameters subject to \eqref{aba}.
More symmetric integral formulas arise from the following
 one-parameter extension of \eqref{gci}.

\begin{proposition}\label{sgl}
For generic parameters subject to the relations 
\begin{equation}\label{lab}\lambda_1\lambda_2\lambda_3=qa_1a_2a_3=q^2/b_1b_2b_3,\end{equation} 
\begin{multline}\label{ggi}
\oint \prod_{j=1}^3\frac{\theta(\lambda_jz)}{(a_jz,b_j/z)_\infty}\frac{dz}{2\pi\ti z}
=\frac{1}{(q)_\infty\theta(\lambda_1/\lambda_2)\prod_{i,j=1}^3(a_ib_j)_\infty}\\
\times\left\{\prod_{j=1}^3\theta(\lambda_1/a_j,\lambda_2b_j)-\frac{\lambda_1}{\lambda_2}\prod_{j=1}^3\theta(\lambda_1b_j,\lambda_2/a_j)\right\}.
\end{multline}
\end{proposition}

To prove Proposition \ref{sgl} we will use \eqref{gci} and the following  fact. 
The proof is a standard elliptic function argument that we include
for completeness. 

\begin{lemma}\label{al}
Let $f$ be an analytic function on $\mathbb C\setminus\{0\}$ that satisfies 
\begin{equation}\label{ct}f(qx)=f(x)/q^3x^6,\qquad f(1/x)=-f(x).\end{equation}
 Moreover, let $c_1$ and $c_2$ be numbers such that $c_1^2,\,c_2^2,c_1c_2,c_1/c_2\notin q^{\mathbb Z}$. Then, $f$ is uniquely determined by the values $f(c_1)$ and $f(c_2)$.
\end{lemma}

\begin{proof}
Let
$$g(x)=x^{-1}\theta(x^2)\left(\frac{\theta(c_2x^\pm)}{c_1^{-1}\theta(c_1^2,c_2c_1^\pm)}\,f(c_1)+\frac{\theta(c_1x^\pm)}{c_2^{-1}\theta(c_2^2,c_1c_2^\pm)}\,f(c_2)\right). $$
We claim that $f(x)=g(x)$ for all $x$.
This is clearly true for $x=c_1$. Since 
$g$ satisfies \eqref{ct},
it is more generally  true for $x^\pm \in c_1 q^{\mathbb Z}$. We also observe that
\eqref{ct} implies that $f$ vanishes for $x^2\in q^{\mathbb Z}$. Hence, the function
$$h(x)=\frac{f(x)-g(x)}{x^{-1}\theta(x^2,c_1x^\pm)} $$ 
is analytic for $x\neq 0$ and satisfies $h(qx)=h(x)$. By Liouville's theorem, $h$ is constant. Since $h(c_2)=0$, we arrive at the desired conclusion.
\end{proof}

\begin{proof}[Proof of \emph{Proposition \ref{sgl}}] We
substitute $\lambda_1=\lambda x$ and $\lambda_2=\lambda/x$ in \eqref{ggi}
and multiply the identity by $x^{-1}\theta(x^2)$. Consider both sides as functions of $x$.
It is easy to see that they satisfy \eqref{ct} (as the contour of integration can be chosen independently of $x$, it is enough to consider the integrand). By Lemma \ref{al}, it then suffices to verify \eqref{ggi} for two generic values of $\lambda_1$. Choosing 
  $\lambda_1=q/b_3$, \eqref{ggi} is reduced to \eqref{gci}. By symmetry, we may take $\lambda_1=q/b_2$ as the other value.
\end{proof}

Let us now replace $a_j$ and $b_j$  in \eqref{ggi} by $a_jq^{k_j}$ and $b_jq^{l_j}$, where $k_j$ and $l_j$ are integers subject to \eqref{klb}. In order to respect \eqref{lab} we also replace $\lambda_3$ by $\lambda_3q^T$.
The resulting identity can be expressed as
\begin{multline}\mathbf K(g)=\frac {(q)_\infty\theta(\lambda_1/\lambda_2)\prod_{i,j=1}^3(a_ib_j)_\infty}{\prod_{j=1}^3\theta(\lambda_1/a_j,\lambda_2b_j)-\frac{\lambda_1}{\lambda_2}\prod_{j=1}^3\theta(\lambda_1b_j,\lambda_2/a_j)}\\
\label{ki}\times\oint g(z)\prod_{j=1}^3\frac{\theta(\lambda_jz)}{(a_jz,b_j/z)_\infty}\frac{dz}{2\pi\ti z},\end{multline}
where  $g$ is as in \eqref{g}. As these functions span $W$, we may draw the following conclusion.

\begin{proposition}\label{cgp}
The functional $\mathbf K$ is given by the integral formula \eqref{ki}, where $\lambda_j$ are generic parameters subject to \eqref{lab}.
\end{proposition}

The special case corresponding to \eqref{gci} is
\begin{multline}\label{gk}\mathbf K(g)=\frac{(q)_\infty}{\theta(b_1\lambda,b_2\lambda)}\prod_{j=1}^3\frac{(b_1a_j,b_2a_j)_\infty}{(q/b_3a_j)_\infty}\\
\times\oint g(z)\frac{(qz/b_3)_\infty\theta(\lambda z,qz/\lambda b_1b_2)}{(a_1z,a_2z,a_3z,b_1/z,b_2/z)_\infty}\frac{dz}{2\pi\ti z}. 
  \end{multline}
 By \eqref{kis}, we may alternatively write
 \begin{multline}\label{gkb}\mathbf K(g)=\frac{(q)_\infty}{\theta(\lambda/a_1,\lambda/a_2)}\prod_{j=1}^3\frac{(a_1b_j,a_2b_j)_\infty}{(q/a_3b_j)_\infty}\\
\times\oint g(z)\frac{(q/a_3z)_\infty\theta(\lambda z,qza_1a_2/\lambda)}{(a_1z,a_2z,b_1/z,b_2/z,b_3/z)_\infty}\frac{dz}{2\pi\ti z}. 
  \end{multline}
We also mention the further specialization $\lambda=a_3$ of \eqref{gk} (or $\lambda=q/b_3$ of \eqref{gkb}),
\begin{multline}\label{aik}\mathbf K(g)=\frac{(q)_\infty\prod_{i,j=1}^2(a_ib_j)_\infty}{(q/a_3b_3)_\infty\prod_{j=1}^2(q/a_3b_j,q/b_3a_j)_\infty}\\
\times\oint g(z)\frac{(qz/b_3,q/a_3z)_\infty\theta(a_1a_2b_3z)}{(a_1z,a_2z,b_1/z,b_2/z)_\infty}\frac{dz}{2\pi\ti z}. 
  \end{multline}
To obtain discrete integral formulas for $\mathbf K$, we replace $a_j$ by $a_jq^{k_j}$ and $b_j$ by $b_jq^{l_j}$ in \eqref{bsi}, where we still assume \eqref{klb}. This leads to the following result.

\begin{proposition}\label{dgp} For generic parameters $\lambda_j$,
the functional $\mathbf K$ can be expressed as
\begin{equation}\label{dk}\mathbf K(g)=C\sum_{k=1}^3X_k\sum_{x=-\infty}^\infty q^xg(\lambda_k q^x)\prod_{j=1}^3\frac{(a_j\lambda_k)_x}{(q\lambda_k/b_j)_x}, \end{equation}
where 
$$C=\frac {\lambda_1\lambda_2\lambda_3}{{(q)_\infty\theta(t\lambda_1\lambda_2\lambda_3)}\prod_{i,j=1}^3(q/a_ib_j)_\infty\prod_{1\leq i<j\leq 3}\lambda_j\theta(\lambda_i/\lambda_j)},$$ 
\begin{align*}
X_1&=\prod_{j=1}^3(q/a_j\lambda_1,q\lambda_1/b_j)_\infty\left\{
\frac{\lambda_2}{\lambda_3}\prod_{j=1}^3\theta(a_j\lambda_3,b_j/\lambda_2)-\frac{\lambda_3}{\lambda_2}\prod_{j=1}^3\theta(a_j\lambda_2,b_j/\lambda_3)\right\},\\
X_2&=\prod_{j=1}^3(q/a_j\lambda_2,q\lambda_2/b_j)_\infty\left\{
\frac{\lambda_3}{\lambda_1}\prod_{j=1}^3\theta(a_j\lambda_1,b_j/\lambda_3)-\frac{\lambda_1}{\lambda_3}\prod_{j=1}^3\theta(a_j\lambda_3,b_j/\lambda_1)\right\},\\
X_3&=\prod_{j=1}^3(q/a_j\lambda_3,q\lambda_3/b_j)_\infty\left\{
\frac{\lambda_1}{\lambda_2}\prod_{j=1}^3\theta(a_j\lambda_2,b_j/\lambda_1)-\frac{\lambda_2}{\lambda_1}\prod_{j=1}^3\theta(a_j\lambda_1,b_j/\lambda_2)\right\}.
 \end{align*} 
\end{proposition}

Note that, in contrast to Proposition \ref{cgp},
  we are not assuming any balancing condition for the parameters $\lambda_j$. 
  
  In analogy with \eqref{dis}, assuming
   that $g$ has no poles at $z\in b_1q^{\mathbb Z_{<0}}$ or $b_2q^{\mathbb Z_{<0}}$,  we may choose $\lambda_1=b_1$ and $\lambda_2=b_2$ in \eqref{dk}. This leads to the expression
  \begin{multline}\label{grs}
  \mathbf K(g)=\frac{1}{b_1\theta(b_2/ b_1)\prod_{j=1}^3(q/a_jb_3)_\infty}\Bigg(b_1(qb_1/b_2,qb_1/b_3,a_1b_2,a_2b_2,a_3b_2)_\infty\\
  \times\sum_{x=0}^\infty
  q^xg(b_1q^x)\prod_{j=1}^3\frac{(a_jb_1)_x}{(qb_1/b_j)_x}-\text{idem}(b_1;b_2)\Bigg),
  \end{multline}
  which can alternatively be obtained from  \eqref{gk} by shrinking the contour to zero, picking up residues at the points $b_1q^{\mathbb Z_{\geq 0}}$ and $b_2q^{\mathbb Z_{\geq 0}}$. 
  If one instead expands the contour in \eqref{gkb}, or substitute $\lambda_1=1/a_1$, $\lambda_2=1/a_2$ in \eqref{dk}, one finds a similar
  identity involving values of $g$ at $a_jq^{\mathbb Z_{\leq 0}}$.

\subsection{Biorthogonal rational functions}
\label{grss}

Let us now consider the limit of Rahman's biorthogonal functions corresponding to the functional $\mathbf K$. There are two essentially different cases, depending on whether the distinguished parameters $b_5$ and $b_6$ in \eqref{rbr} tend to the same or distinct limits in $\{0,\infty\}$.

For the first case, we write
\begin{multline*}q_n(z;b_1,b_2,b_3,a_1,a_2,a_3;q)\\
\begin{split}&=\lim_{c\rightarrow 0} \,Q_n\left(\frac{(cz)+(cz)^{-1}}{2};b_1c,b_2c,b_3c,a_1/c,a_2/c,a_3/c;q\right) \\
&=(a_1b_1,1/a_3b_1)_n\,{}_4\phi_3\left(\begin{matrix}q^{-n},q/a_2b_2,q/a_2b_3,b_1/z\\ a_1b_1,q^{1-n}a_3b_1,q/a_2z\end{matrix};q\right).
\end{split}\end{multline*}
This function is symmetric in the parameters $(b_1,b_2,b_3)$. Though we label $q_n$ by seven parameters (counting $q$),  it is effectively a five-parameter system since  \eqref{aba} is assumed and we may scale $z$, $a_j^{-1}$ and $b_j$ by a fixed constant.
If $r_n$ denotes the function $q_n$ with $a_2$ and  $a_3$ interchanged, then it follows from \eqref{rbr} that
$$\mathbf K(q_mr_n)=\delta_{mn}\,q^{-n}(q,a_1b_1,a_1b_2,a_1b_3)_n.  $$

Applying \eqref{kis}, we also find that if
\begin{align*}\tilde q_n(z)&=q_n(1/z;a_1,a_2,a_3,b_1,b_2,b_3;q)\\
&=(a_1b_1,1/a_1b_3)_n\,{}_4\phi_3\left(\begin{matrix}q^{-n},q/a_2b_2,q/a_3b_2,a_1z\\ a_1b_1,q^{1-n}a_1b_3,qz/b_2\end{matrix};q\right)
\end{align*}
and $\tilde r_n$ denotes  $\tilde q_n$ with $b_2$ and $b_3$ interchanged, then
$$\mathbf K(\tilde q_m\tilde r_n)=\delta_{mn}\,q^{-n}(q,a_1b_1,a_2b_1,a_3b_1)_n.  $$
For reasons that will be clear in \S \ref{sis} we think of this as a separate system, even though it is equivalent to $(q_n,r_n)$ by a change of parameters.

 The
biorthogonal system $(q_n,r_n)$ appears in the scheme of  van de Bult and Rains \cite{br} with the label
$3100v2$. The biorthogonality relations described there correspond to \eqref{gk} and \eqref{grs}, 
 possibly after permuting the parameters and applying the symmetry \eqref{kis}.
Our expressions \eqref{ki} and \eqref{dk} unify and generalize these relations. 
Just as for Rahman's functions, in the continuous case the contour of integration depends on $m$ and $n$, whereas
the discrete measures are fixed.

For the second case, we write
\begin{multline*}s_n(z;b_1,b_2,b_3,a_1,a_2,a_3;q)\\
\begin{split}&=\lim_{c\rightarrow 0}(-1)^nq^{-\binom n2}c^{2n} Q_n\left(\frac{(cz)+(cz)^{-1}}{2};b_1c,b_2c,a_1/c,a_2/c,a_3/c,b_3c;q\right) \\
&=\frac{(a_1b_1,a_2b_1)_n}{(b_1b_3)^n}\,{}_4\phi_3\left(\begin{matrix}q^{-n},q/a_3b_2,q^n/a_3b_3,b_1/z\\ a_1b_1,a_2b_1,q/a_3z\end{matrix};q\right),
\end{split}\end{multline*}
\begin{multline*}t_n(z;b_1,b_2,b_3,a_1,a_2,a_3;q)\\
\begin{split}&=\lim_{c\rightarrow 0} Q_n\left(\frac{(cz)+(cz)^{-1}}{2};a_1/c,b_1c,b_2c,a_2/c,b_3c,a_3/c;q\right) \\
&=(q^{-1}a_2b_3)^n{(a_1b_1,a_1b_2)_n}\,{}_4\phi_3\left(\begin{matrix}q^{-n},q/a_2b_3,q^n/a_3b_3,a_1z\\ a_1b_1,a_1b_2,qz/b_3\end{matrix};q\right).
\end{split}\end{multline*}
Again, these are effectively five-parameter systems. Both
 $s_n$ and $t_n$ are symmetric in the pairs $(a_1,a_2)$ and $(b_1,b_2)$. We obtain from \eqref{rbr} the biorthogonality 
\begin{equation}\label{stb}\mathbf K(s_mt_n)=\delta_{mn}\,(q^{-1}a_1a_2)^n\frac{1-q^{-1}a_1a_2b_1b_2}{1-q^{2n-1}a_1a_2b_1b_2}\frac{(q)_n\prod_{i,j=1}^2(a_ib_j)_n}{(q^{-1}a_1a_2b_1b_2)_n}.  \end{equation}
In this case, it is natural to use the expression \eqref{aik} for $\mathbf K$, as the poles of $s_m$ and $t_n$ are then 
situated at zeroes of the biorthogonality measure, so that  one may use the same contour of integration  for all $m$ and $n$.
With the biorthogonality written in this form, \eqref{stb} is due to
Al-Salam and Ismail \cite{ai}.
In the classification of \cite{br}, the system $(s_n,t_n)$ is denoted $2200vv$
(in Appendix A.2 it appears with a typo as $2200vp$). The expressions  \eqref{ki} and \eqref{dk} for $\mathbf K$ generalize the 
 biorthogonality measures found in \cite{ai} and \cite{br}.

\section{Superconformal indices}
\label{sis}

As was discussed in the introduction, when considering integral evaluations related to dualities for three-dimensional supersymmetric quantum field theories,  \eqref{gi} appears as a
top level beta integral. 
As before, the integration is over a positively oriented contour separating the geometric sequences of poles tending to $0$ from those tending to $\infty$; note that the contour necessarily depends on $x$.
Besides the quantum field theory interpretation, \eqref{gi} also appears as
the star-triangle relation for a solvable lattice model introduced in \cite{gs}. 

Let us  replace $b_j$ in \eqref{gi} by $b_jq^{k_j}$ and $N_j$ by $N_j+l_j-k_j$, where $k_j$ and $l_j$ are integers with $k_1+\dots+k_6=l_1+\dots+l_6=0$. 
After simplification, this results in the identity
\begin{multline*}\sum_{x=-\infty}^\infty\oint \frac{(1-q^x z^2)(1-q^xz^{-2})}{q^xz^{6x}}\prod_{j=1}^6\frac{(q^{1+x/2}/b_jz,q^{1-x/2}z/b_j)_\infty}{(q^{N_j+x/2}b_jz,q^{N_j-x/2}b_j/z)_\infty}\\
\times\prod_{j=1}^6{(b_j(q^{-x/2}z)^{\pm})_{k_j}(b_jq^{N_j}(q^{x/2}z)^\pm)_{l_j}}\, \frac{dz}{2\pi\ti z}\\
=\frac 2{\prod_{j=1}^6q^{\binom{k_j}2+\binom{l_j}2+\binom{N_j}2+N_jl_j}b_j^{k_j+l_j+N_j}}
\prod_{1\leq i<j\leq 6}
\frac{(q/b_ib_j)_\infty(b_ib_j)_{k_i+k_j}(b_ib_jq^{N_i+N_j})_{l_i+l_j}}{(b_ib_jq^{N_i+N_j})_\infty}.
\end{multline*}
This can be interpreted in terms of the functional \eqref{id}. Let $V$ be the space introduced in \S \ref{rfs} and $V'$ the space obtained from $V$ through replacing the parameters $b_j$ with $b_jq^{N_j}$.
Denoting the Rahman functional on $V'$ by $\mathbf J'$, we obtain the following result.

\begin{proposition}\label{dpp} For $f\in V$ and $g\in V'$,
\begin{multline}\label{dc}\sum_{x=-\infty}^\infty\oint\frac{(1-q^x z^2)(1-q^xz^{-2})}{q^xz^{6x}}\prod_{j=1}^6\frac{(q^{1+x/2}/b_jz,q^{1-x/2}z/b_j)_\infty}{(q^{N_j+x/2}b_jz,q^{N_j-x/2}b_j/z)_\infty}\\
\times 
f\left(\frac{q^{-x/2}z+q^{x/2}z^{-1}}2\right)g\left(\frac{q^{x/2}z+q^{-x/2}z^{-1}}2\right)\frac{dz}{2\pi\ti z}\\
=\frac 2{\prod_{j=1}^6q^{\binom{N_j}2} b_j^{N_j}}\prod_{1\leq i<j\leq 6}
\frac{(q/b_ib_j)_\infty}{(b_ib_jq^{N_i+N_j})_\infty}\,\mathbf J(f)\mathbf J'(g).
\end{multline}
\end{proposition}

We find it remarkable that the right-hand side of \eqref{dc} factors, even though the 
variables of $f$ and $g$ on the left are coupled. We refer to this
as the \emph{decoupling phenomenon}.
Just as for Spiridonov's beta integral discussed in the introduction, it can be used to construct two-index biorthogonal systems.
To this end, let $Q_n$ and $R_n$ be as in \eqref{rqr}. Moreover,  if $(c_1,\dots,c_6)$ is any permutation of $(b_1q^{N_1},\dots,b_6q^{N_6})$, let
\begin{align*}
Q_n'&=Q_n\left(\frac{z+z^{-1}}2;c_1,c_2,c_3,c_4,c_5,c_6\right),\\
R_n'&=Q_n\left(\frac{z+z^{-1}}2;c_1,c_2,c_3,c_4,c_6,c_5\right).
\end{align*}
Combining \eqref{rbr} and Proposition \ref{dpp} then gives the following result.

\begin{corollary}\label{bpc}
  In the notation above,
\begin{multline*}\sum_{x=-\infty}^\infty\oint\frac{(1-q^x z^2)(1-q^xz^{-2})}{q^xz^{6x}}\prod_{j=1}^6\frac{(q^{1+x/2}/b_jz,q^{1-x/2}z/b_j)_\infty}{(q^{N_j+x/2}b_jz,q^{N_j-x/2}b_j/z)_\infty}\\
\times 
(Q_{n_1}R_{m_1})\left(\frac{q^{-x/2}z+q^{x/2}z^{-1}}2\right)(Q_{n_2}'R_{m_2}')\left(\frac{q^{x/2}z+q^{-x/2}z^{-1}}2\right)\frac{dz}{2\pi\ti z}\\
=\delta_{n_1m_1}\delta_{n_2m_2}\frac 2{\prod_{j=1}^6q^{\binom{N_j}2} b_j^{N_j}}\prod_{1\leq i<j\leq 6}
\frac{(q/b_ib_j)_\infty}{(b_ib_jq^{N_i+N_j})_\infty}\\
\times\frac{1-1/b_5b_6}{1-q^{2n_1}/b_5b_6}\frac{(q)_{n_1}\prod_{1\leq i<j\leq 4}(b_ib_j)_{n_1}}{q^{n_1}(1/b_5b_6)_{n_1}}
\frac{1-1/c_5c_6}{1-q^{2n_2}/c_5c_6}\frac{(q)_{n_2}\prod_{1\leq i<j\leq 4}(c_ic_j)_{n_2}}{q^{n_2}(1/c_5c_6)_{n_2}}
.
\end{multline*}
\end{corollary}

Up to permutation of the parameters, there are  three non-equivalent cases of Corollary \ref{bpc},  corresponding to whether
none, one or  both of the numbers $c_5,c_6$ are contained in $\{b_5q^{N_5},b_6q^{N_6}\}$.

%

At the Saalsch\"utz level, we have the superconformal index identity \cite{ga}
\begin{multline}\label{sp}\sum_{x=-\infty}^\infty\oint\prod_{j=1}^3\frac{(q^{1+x/2}/a_jz,q^{1-x/2}z/b_j)_\infty}{(q^{M_j+x/2}a_jz,q^{N_j-x/2}b_j/z)_\infty}\left(-\frac{q^{1/2}}{tz^3}\right)^x\frac{dz}{2\pi\ti z}\\
=\frac{1}{\prod_{j=1}^3q^{\binom{M_j}2+\binom{N_j}2}a_j^{M_j}b_j^{N_j}}\prod_{i,j=1}^3\frac{(q/a_ib_j)_\infty}{(a_ib_jq^{M_i+N_j})_\infty}, \end{multline}
where the parameters satisfy \eqref{aba}, $t$ is given by \eqref{t} and the integers $M_j$ and $N_j$ satisfy
 $M_1+M_2+M_3=N_1+N_2+N_3=0$. In \cite{ga} this is stated under the additional assumption  $t=q^{1/2}$, but
 the general case follows immediately after  rescaling $z$ and the parameters. 
  
 We replace the parameters in \eqref{sp} by $a_j\mapsto a_jq^{k_j}$, $b_j\mapsto b_jq^{l_j}$, $M_j\mapsto M_j+m_j-k_j$, $N_j\mapsto N_j+n_j-l_j$, where the integral shifts satisfy
  $$k_1+k_2+k_3=-l_1-l_2-l_3=m_1+m_2+m_3=-n_1-n_2-n_3=T. $$
 The resulting identity can be written
  \begin{multline*}\sum_{x=-\infty}^\infty\oint\prod_{j=1}^3\frac{(q^{1+x/2}/a_jz,q^{1-x/2}z/b_j)_\infty}{(q^{M_j+x/2}a_jz,q^{N_j-x/2}b_j/z)_\infty}\left(-\frac{q^{1/2}}{tz^3}\right)^x\\
  \times z^{-2T}\prod_{j=1}^3(q^{-x/2}a_jz)_{k_j}(q^{x/2}b_j/z)_{l_j}(q^{M_j+x/2}a_jz)_{m_j}(q^{N_j-x/2}b_j/z)_{n_j} \,\frac{dz}{2\pi\ti z}\\
=\frac{q^{2\binom T2} (t^2q^{M_1+M_2+M_3})^T}{\prod_{j=1}^3q^{\binom {k_j}2+\binom{l_j}2+\binom{m_j}2+\binom{n_j}2+\binom{M_j}2+\binom{N_j}2+m_jM_j+n_jN_j}a_j^{k_j+m_j+M_j}b_j^{l_j+n_j+N_j}}\\
\times\prod_{i,j=1}^3\frac{(q/a_ib_j)_\infty(a_ib_j)_{k_i+l_j}(a_ib_jq^{M_i+N_j})_{m_i+n_j}}{(a_ib_jq^{M_i+N_j})_\infty}. \end{multline*}
If we now let $\mathbf K$ and $W$ be the functional and space of rational functions introduced in \eqref{kds} and $\mathbf K'$, $W'$ the objects obtained from these  by replacing the parameters $a_j$ by $a_jq^{M_j}$ and $b_j$ by $b_jq^{N_j}$, we may draw the following conclusion.
  
  \begin{proposition}
 For $f\in W$ and $g\in W'$, 
  \begin{multline*}\sum_{x=-\infty}^\infty\oint\prod_{j=1}^3\frac{(q^{1+x/2}/a_jz,q^{1-x/2}z/b_j)_\infty}{(q^{M_j+x/2}a_jz,q^{N_j-x/2}b_j/z)_\infty}\left(-\frac{q^{1/2}}{tz^3}\right)^x\, f(q^{-x/2}z)\,g(q^{x/2}z)\,\frac{dz}{2\pi\ti z}\\
=\frac{1}{\prod_{j=1}^3q^{\binom{M_j}2+\binom{N_j}2}a_j^{M_j}b_j^{N_j}}\prod_{i,j=1}^3\frac{(q/a_ib_j)_\infty}{(a_ib_jq^{M_i+N_j})_\infty}\,\mathbf K(f)\mathbf K'(g). \end{multline*}
  \end{proposition}

We can now use the results of \S \ref{grss} to construct corresponding biorthogonal systems. This can be done in many ways, since we may for each of the spaces $W$ and  $W'$ work with the system $(q_n,r_n)$, $(\tilde q_n,\tilde r_n)$ or $(s_n,t_n)$. Moreover, 
permuting the parameters gives cases when both, one or none of the two functions chosen for $W$ have poles at the same geometric sequences as the functions chosen for $W'$.
In total, this gives ten essentially non-equivalent possibilities.
  We refrain from writing these down explicitly.


 \end{document}